\documentclass[reqno,11pt]{amsart}

\usepackage{amsmath, amssymb, amsfonts, amsthm}  % AMS 数学包
\usepackage{mathrsfs}    % 花体字母
\usepackage{bbm}         % 黑板粗体 \mathbbm
\usepackage{dsfont}      % \mathds
\usepackage{latexsym}    % 经典符号
\usepackage{euscript}    % \EuScript
\usepackage{hyperref}
\usepackage{empheq} 
\usepackage{geometry}
\allowdisplaybreaks      % 允许公式跨页断开

\usepackage{xcolor}      % 颜色支持
\usepackage{enumitem}    % 可定制列表
\usepackage{enumerate}   % enumerate 扩展
\makeatletter

\newcommand{\Rmnum}[1]{\expandafter\@slowromancap\romannumeral #1@}
\makeatother

\numberwithin{equation}{section}  % 公式按 section 编号

\newtheorem{theorem}{Theorem}[section]
\newtheorem{lemma}[theorem]{Lemma}

\newtheorem{proposition}[theorem]{Proposition}
\newtheorem{remark}[theorem]{Remark}

\begin{document}
	\title[Exact approximation order of real numbers in Cantor series expansions]{Exact approximation order of real numbers in Cantor series expansions}

	\author[W. Cheng]{Wanjin Cheng}
	\address[Wanjin Cheng]{School of Mathematics, South China University of Technology, Guangzhou, 510640, China}
	\email{chengwj0227@163.com}
	\author[X. Zhang]{Xinyun Zhang$^{*}$}
	\address[Xinyun Zhang]{School of Mathematics and Information Science, Nanchang Hangkong University, Nanchang, 330063, China}
	\email[corresponding author]{xinyunzhangnc@163.com}
	
	\thanks{$^{*}$ Corresponding author.}
\begin{abstract}
	Let $Q=\{q_n\}_{n\geq1}$ be a sequence of integers with $q_n\geq2$ for every $n\geq1$. Every $x\in[0,1)$ admits a Cantor series expansion of the form
	\[
	x
	=
	\frac{\varepsilon_1(x)}{q_1}
	+
	\frac{\varepsilon_2(x)}{q_1q_2}
	+\cdots+
	\frac{\varepsilon_n(x)}{q_1q_2\cdots q_n}
	+\cdots.
	\]
	In this paper, we study exact approximation  of real numbers by the $n$-th partial sums of their Cantor series expansions. More precisely, for a non-increasing function $\psi:\mathbb{N}\to(0,\infty)$ satisfying $\psi(n)\to0$, let $E_c(\psi)$ be the set of points that are $\psi$-approximable by their $n$-th partial sums but are not $b\psi$-approximable for any $0<b<1$. We determine the Hausdorff dimension of $E_c(\psi)$.
%	\[
%	\dim_H E_c(\psi)
%	=
%	\frac{1}{1+\alpha},
%	\qquad
%	\alpha
%	=
%	\liminf_{n\to\infty}
%	\frac{-\log\psi(n)}{\log(q_1\cdots q_n)}.
%	\]
\end{abstract}
	\maketitle

	\section{Introduction}

	Metric Diophantine approximation is concerned with the quantitative study of how well real numbers can be approximated by rational numbers.
	This originates in the work of Khintchine \cite{Khinchin24} who studied the Lebesgue measure of the sets
	\[
	W(\psi) :=
	\Bigl\{
	x \in \mathbb{R} :
	\bigl|x - \tfrac{p}{q}\bigr| < \psi(q)
	\ \text{for infinitely many } (p,q)\in\mathbb{Z}\times\mathbb{N}
	\Bigr\}.
	\]
	For $\tau\ge2$, let
		$\psi_\tau(q)=q^{-\tau}$ and write $W(\tau)=W(\psi_\tau)$. Jarn\'ik \cite{Jarnik29} and Besicovitch \cite{Besicovitch34} independently proved that $\dim_HW(\tau)=\frac{2}{\tau}$. These results form the foundation of the metric theory of Diophantine approximation.
	%		which has since developed into a rich theory for higher dimensional, for example \cite{Bovey86, Dodson92, Beresnevich06, Beresnevich10}.
	Jarn\'ik \cite{Jarnik31} also introduced the exact approximation set 
	\[
	\mathrm{Exact}(\psi)
	:=
	W(\psi)\setminus
	\bigcup_{0<c<1} W(c\psi),
	\]
	and showed that it is non-empty when $\psi(q)=o(q^{-2})$. Set
	$\operatorname{Exact}(\tau)=\operatorname{Exact}(\psi_\tau)$.
	G\"uting \cite{Guting69} further proved $\dim_H\mathrm{Exact}(\tau) =\frac{2}{\tau}$ for $\tau\geq 2$. Bugeaud \cite{Bugeaud03} later showed that if \(q\mapsto q^{2}\psi(q)\) is non-increasing and \(\sum_{q=1}^{\infty} q\psi(q)<\infty\), then
	\[
	\dim_H \mathrm{Exact}(\psi)
	=
	\dim_H W(\psi)
	=
	\frac{2}{\lambda},
	\quad
	\lambda
	=
	\liminf_{q\to\infty}
	\frac{-\log\psi(q)}{\log q}.
	\] Further results on exact approximation can be found in
		\cite{Bugeaud08,Bugeaud12,Bandi23}.

	Diophantine approximation has become closely connected with representations of real numbers, as many Diophantine properties can be characterized via symbolic digit expansions. Consequently, various Diophantine approximation problems associated with continued fractions \cite{Fan09}, $\beta$-expansions \cite{Fang20,Zhang24}, and other expansion systems have been extensively studied. In this paper, we focus on a non-autonomous dynamical system, namely the Cantor series expansion, which was first introduced by Cantor \cite{Cantor69}. Various problems concerning this non-autonomous dynamical system, including distribution phenomena and shrinking target problems, have been investigated in \cite{Airey15,Fishman15,Sun17}.

	Let \(Q = \{q_n\}_{n \ge 1}\) be a sequence of integers with \(q_n \ge 2\) for all \(n \ge 1\). For each \(n \ge 1\), define the transformation \(T_{Q,n} : [0,1) \to [0,1)\) by
	\[
	T_{Q,n}(x) = q_n x - \lfloor q_n x \rfloor.
	\]
	We also define the composition
	\[
	T_Q^{n}(x) = T_{Q,n} \circ \cdots \circ T_{Q,1}(x) = q_1 q_2 \cdots q_n x - \lfloor q_1 q_2 \cdots q_n x \rfloor.
	\]
	Set $\varepsilon_1(x) = \lfloor q_1 x \rfloor$ and $\varepsilon_n(x) = \big\lfloor q_n T_Q^{n-1}(x) \big\rfloor$ for $n \ge 2$,
	then every \(x \in [0,1)\) admits a Cantor series expansion of the form
	\begin{equation}\label{eq:Cantor_expansion}
		x =
		\frac{\varepsilon_1(x)}{q_1}
		+ \frac{\varepsilon_2(x)}{q_1 q_2}
		+ \cdots
		+ \frac{\varepsilon_n(x)}{q_1 q_2 \cdots q_n}
		+ \cdots.
	\end{equation}
	%	where \(\varepsilon_n(x) \in \{0,1,\ldots,q_n-1\}\) for all \(n \ge 1\), and \(\varepsilon_n(x) = q_n - 1\) occurs infinitely often. The sequence \((\varepsilon_1(x), \varepsilon_2(x), \ldots)\) is called the \(Q\)-Cantor series expansion of \(x\). The expansion is said to be finite if there exists \(m \ge 1\) such that \(\varepsilon_m(x) > 0\) and \(\varepsilon_n(x) = 0\) for all \(n > m\); otherwise, it is called infinite. It is easily seen that any number with a finite \(Q\)-Cantor series expansion is rational. 
	%Cantor series expansions generalize the classical $b$-ary expansions by allowing a varying base sequence and infinitely many possible digit structures. From a dynamical systems viewpoint, they generate a nonautonomous system. Consequently, they have been extensively studied in the literature \cite{Airey15, Fishman15, Sun17}.
	
	%where $\lambda:=\liminf\limits_{n\rightarrow+\infty}\frac{-\log\psi(n)}{\log n}$.}
%While the theory of exact approximation in the classical setting has been well developed, much less is known in the context of $\beta$-expansions.
Let $\omega_n(x)$ denote the sum of first $n$ terms of formula in (\ref{eq:Cantor_expansion}), the approximation error $x-\omega_n(x)$ is naturally governed by the growth of $q_1q_2\cdots q_n$. Recently, Ma et al. \cite{Ma26} studied the approximation error and proved that
$$
\lim_{n \to +\infty} \frac{\log \bigl( x - \omega_n(x) \bigr)}{\log (q_1\cdots q_n)}
= -1 \quad \text{for Lebesgue almost every}\ x\in[0,1).$$
Moreover, if $\varphi$ is non-decreasing and satisfies
$\lambda := \liminf\limits_{n \to +\infty} \frac{\varphi(n)}{\log q_1\cdots q_n} >1,$
then
\[
\dim_H \Bigl\{ x \in [0,1):
\liminf_{n \to +\infty} \frac{\log \big( x - \omega_n(x)\big)}{\varphi(n)} = -1 \Bigr\}
= \frac{1}{\lambda}.
\]

Motivated by Jarn\'ik's pioneering work and subsequent developments on exact approximation, we investigate the analogue of exact approximation problem in the setting of Cantor series expansions. Let
	\(\psi:\mathbb N\to(0,\infty)\) be non-increasing and satisfy
	\(\psi(n)\to0\). Define
$$W_{c}(\psi):=\Big\{x\in[0,1): x-\omega_n(x)<\frac{\psi(n)}{q_1\cdots q_n} ~\text{for infinitely many} ~n\in\mathbb{N}\Big\},$$
and  the corresponding exact approximation set
\[E_{c}(\psi):=W_{c}(\psi)\backslash
\bigcup_{0<b<1}W_{c}(b\psi).
\]
Our main result gives the Hausdorff dimension of $E_{c}(\psi)$.
\begin{theorem}\label{mainthm}
	Let $Q=\{q_n\}_{n\ge 1}$ be a sequence of integers with $q_n\ge2$ for all $n\in\mathbb{N}$ and satisfying
	\begin{equation}\label{request}
		\lim_{n\rightarrow\infty}
		\frac{\log q_{n}}{\log(q_1\cdots q_n)}=0.
	\end{equation}
	Let $\psi:\mathbb{N}\rightarrow\mathbb{R}_{+}$ be a non-increasing function with $\psi(n)\rightarrow0$ as $n\rightarrow+\infty$.
	Then
	$$\dim_{H}E_{c}(\psi)=\frac{1}{1+\alpha},\quad \text{where}\ \alpha= \liminf_{n \to +\infty} \frac{-\log\psi(n)}{\log (q_1 \cdots q_n)}.$$
	%where $\alpha_i = \liminf\limits_{n \to \infty} \frac{-\log_{\beta_i} \psi_i(n)}{ n}$.
\end{theorem}
\begin{remark}Under the assumptions of Theorem~\ref{mainthm}, we have $ \dim_{H}W_{c}(\psi)=\frac{1}{1+\alpha}$. Thus \(E_c(\psi)\) and \(W_c(\psi)\) have the same Hausdorff dimension.
\end{remark}
%\begin{Corollary}
%Let $d\geq2$ be an integer. For each $1 \leq i \leq d$, let $\beta_i>1$ be real numbers and let $\psi_i : \mathbb{R}_{>0} \to %\mathbb{R}_{>0}$
%be non-increasing functions with $\psi_i(n)\rightarrow0$ as $n\rightarrow\infty$. Then we have
%\[
%\dim_\mathcal{H} \big( W_{\beta_i}(\psi_1) \times \cdots \times W_{\beta_d}(\psi_d) \big)
%= \min_{1 \leq i \leq d} \Big\{ d - 1 + \dim_\mathcal{H} W_{\beta_i}(\psi_i) \Big\}
%\]
%where $\lambda_i = \liminf\limits_{n \to \infty} \frac{-\log_{\beta_i} \psi_i(n)}{ n}$.
%\end{Corollary}

\section{Preliminaries}
We briefly recall the definitions of Hausdorff measure and Hausdorff dimension that will be used throughout the paper. For standard references, see \cite{Falconer14,Mattila95}.

For a nonempty set \(U\subset\mathbb{R}^d\), let $|U|:=\operatorname{diam}(U)$.
Given \(E\subset\mathbb{R}^d\) and \(\delta>0\), a countable family \(\{U_i\}_{i\geq1}\) is called a \(\delta\)-cover of \(E\) if
\[
E\subseteq\bigcup_{i=1}^{\infty}U_i
\qquad\text{and}\qquad
|U_i|\leq\delta
\quad\text{for every }i\geq1.
\]
For \(s\geq0\), the \(s\)-dimensional Hausdorff content of \(E\) at scale \(\delta\) is defined by
\[
\mathcal{H}_{\delta}^{s}(E)
:=
\inf\left\{
\sum_{i=1}^{\infty}|U_i|^{s}:
\{U_i\}_{i\geq1}\text{ is a }\delta\text{-cover of }E
\right\}.
\]
Since \(\mathcal{H}_{\delta}^{s}(E)\) is non-decreasing as \(\delta\downarrow0\), the \(s\)-dimensional Hausdorff measure of \(E\) is given by
\[
\mathcal{H}^{s}(E)
:=
\lim_{\delta\to0}\mathcal{H}_{\delta}^{s}(E).
\]
The Hausdorff dimension of \(E\) is then
\[
\dim_{H}E
:=
\inf\left\{s\geq0:\mathcal{H}^{s}(E)=0\right\}
=
\sup\left\{s\geq0:\mathcal{H}^{s}(E)=\infty\right\}.
\]

To establish lower bounds for Hausdorff dimension, we shall use the following standard form of the mass distribution principle.

\begin{proposition}[Mass Distribution Principle {\cite{Falconer14}}]
	\label{Mdp}
	Let \(E\subset\mathbb{R}^{d}\) be a Borel set, and let \(\mu\) be a finite Borel measure supported on \(E\) with \(\mu(E)>0\). Suppose that there exist constants \(C>0\) and \(r_{0}>0\) such that
	\[
	\mu\bigl(B(x,r)\bigr)\leq C r^{s}
	\]
	for every \(x\in\mathbb{R}^{d}\) and every \(0<r<r_{0}\). Then
	\[
	\mathcal{H}^{s}(E)\geq\frac{\mu(E)}{C}.
	\]
	In particular, $\dim_{H}E\geq s$.
\end{proposition}

We now present some basic properties of the Cantor series expansion. Let $Q = \{q_k\}_{k \ge 1}$ be a sequence of positive integers with $q_k \ge 2$ for all $k \ge 1$.
For each integer $n \ge 1$, denote
\[
D_n
=
\left\{
(\varepsilon_1,\cdots,\varepsilon_n) \in \mathbb{N}^n :
0 \le \varepsilon_k \le q_k - 1,\; 1 \le k \le n
\right\}.
\]
It is clear that $\# (D_n)= q_1 \cdots q_n$, where ``$\#$'' denotes the cardinality of a finite set.

For any word $(\varepsilon_1,\varepsilon_2,\ldots,\varepsilon_n)\in D_n$, a cylinder set \(I(\varepsilon_1,\varepsilon_2,\ldots,\varepsilon_n)\) of order \(n\) is defined by
\[
I(\varepsilon_1,\varepsilon_2,\ldots,\varepsilon_n)
:=
\{x \in [0,1) : \varepsilon_i(x) = \varepsilon_i,\ \text{for } i=1,2,\ldots,n\}.
\]
Sometimes we simply write \(I_n\) to denote a general cylinder of order \(n\) without specifying the digits \((\varepsilon_1,\varepsilon_2,\ldots,\varepsilon_n)\), and write \(|I_n|\) for its length. By the definition of the \(Q\)-Cantor series expansion, it follows that
\[
I(\varepsilon_1,\cdots,\varepsilon_n)
=
\bigg[
\sum_{i=1}^n \frac{\varepsilon_i}{q_1 \cdots q_i},
\;
\sum_{i=1}^n \frac{\varepsilon_i}{q_1 \cdots q_i}
+
\frac{1}{q_1 \cdots q_n}
\bigg),
\]
which is a left-closed and right-open interval of length
\[
|I(\varepsilon_1,\cdots,\varepsilon_n)|
=
\frac{1}{q_1 \cdots q_n}.
\]
For each $n\ge1$, the cylinders of order $n$ form a partition of $[0,1)$. That is,
\begin{equation*}
	[0,1)
	=
	\bigcup_{(\varepsilon_1,\cdots,\varepsilon_n) \in D_n} I(\varepsilon_1,\cdots,\varepsilon_n),
\end{equation*}
where the union on the right-hand side is disjoint.

Finally, we introduce the following modified mass distribution principle adapted to $Q$-Cantor series expansion.
\begin{lemma}[\cite{Ma26}]\label{3.3}
	Let $Q=\{q_k\}_{k\geq1}$ be a sequence of integers with $q_k\geq2$ for all $k\in\mathbb{N}$ and satisfying  (\ref{request}). Let $\mu$ be a measure supported on a Borel measurable set $E\subseteq [0,1)$ with $\mu(E)>0$. Assume that for some $s>0$, there exists a constant $c > 0$  such that
	\begin{equation*}%\label{e3.5}
		\mu(I_n) \leq c|I_n|^s,
	\end{equation*}
	for all cylinders $I_n$. Then $\dim_HE\geq s$.
\end{lemma}

\section{Proof of theorem \ref{mainthm}}
\subsection{The upper bound} 
\begin{proposition}
	Let $Q=\{q_n\}_{n\ge1}$ be a sequence of integers with $q_n\ge2$ for all $n\in\mathbb{N}$. Then 
	$$\dim_H E_c(\psi)\le \frac{1}{1+\alpha},\quad where~\alpha= \liminf_{n \to +\infty} \frac{-\log\psi(n)}{\log (q_1 \cdots q_n)}. $$
\end{proposition}
\begin{proof}
	By the definition of $E_c(\psi)$, one has 
	\begin{align*}
		E_c(\psi)\subset W_c(\psi)&=\Big\{x\in[0,1): x-\omega_n(x)<\frac{\psi(n)}{q_1\cdots q_n} ~\text{for infinitely many} ~n\in\mathbb{N}\Big\}\\
		&=\bigcap_{N=1}^{\infty}\bigcup_{n=N}^{\infty}\Big\{x\in[0,1): x-\omega_n(x)<\frac{\psi(n)}{q_1\cdots q_n}\Big\}\\
		&\subset\bigcap_{N=1}^{\infty}\bigcup_{n=N}^{\infty}
		\bigcup_{\substack{0 \leq \varepsilon_n < q_i \\ 1 \leq i \leq n}}
		B\Big(\frac{\varepsilon_1}{q_1}+\cdots
		\frac{\varepsilon_n}{q_1\cdots q_n},\frac{\psi(n)}{q_1\cdots q_n}\Big)
	\end{align*}
	By the definition of Hausdorff measure $\mathcal{H}^s$, we have 
	\begin{align*}
		\mathcal{H}^s\big(E_c(\psi)\big)&\le
		\liminf_{N\rightarrow+\infty}\sum_{n=N}^{+\infty}
		(q_1\cdots q_n)\Big(\frac{2\psi(n)}{q_1\cdots q_n}\Big)^s\\
		&=2^s\lim_{N\rightarrow+\infty}\sum_{n=N}^{+\infty}
		\psi(n)^s(q_1\cdots q_n)^{1-s}.
	\end{align*}
	Recall that $\alpha= \liminf\limits_{n \to +\infty} \frac{-\log\psi(n)}{\log (q_1 \cdots q_n)},$ then for any $\eta>0$, there exists $N\in\mathbb{N}$ such that for $n\ge N$, one has 
	$$\frac{-\log \psi(n)}{\log(q_1\cdots q_n)}\ge \alpha-\eta
	\Rightarrow \psi(n)\le 
	(q_1\cdots q_n)^{-(\alpha-\eta)}.$$
	Hence, 
	$$\psi(n)^s(q_1\cdots q_n)^{1-s}\le 
	(q_1\cdots q_n)^{1-s(1+\alpha-\eta)}.$$
	Therefore, for any $s$ with
	$$s>\frac{1}{1+\alpha-\eta},$$
	one has $\mathcal{H}^s(E_c(\psi))=0,$ by the arbitrariness of $\eta>0$, we have $$\dim_H E_c(\psi)\le \frac{1}{1+\alpha}.$$
\end{proof}

\subsection{The lower bound}
We now turn to the lower bound. If $\alpha=\infty$, then
the desired lower bound is immediate. Once the lower bound has been established for an arbitrary $\alpha>0$, the case $\alpha=0$ follows by letting
$\alpha \to0$.
Therefore, throughout this section, we assume
that $0<\alpha<\infty$.

We begin with constructing a suitable Cantor subset of $E_c(\psi)$ and define a mass distribution on this subset. In order to attain the exact approximation order, we must select the digits carefully, the long zero blocks must be avoided. Then we need to check carefully that the constructed Cantor subset contained in $E_{c}(\psi).$
\subsubsection{Construction of the Cantor subset}
Let $s:=\frac{1}{1+\alpha}.$ 
Fix $0<\epsilon<1$ sufficiently small, we can choose a large integer $M$ and a number $\eta$ such that
\begin{equation}\label{eta-M-alpha}
	%\label{Meta}
	M=\big\lceil\frac{12}{\epsilon}\big\rceil+1~\text{and}~\eta<\min\Big\{\frac{\alpha}{2M(4+\alpha)}, \frac{\alpha}{4(1+2\alpha)}, \frac{\epsilon}{4M}\Big\}
\end{equation}

For convenience, we denote $Q_n:=q_1\cdots q_n$. We fix a sequence $\big\{\delta_{k}\big\}_{k\ge1}$ of positive numbers with $\delta_{k}\rightarrow 0$ as $k\rightarrow\infty.$ Since $\psi(n)$ tends to zero as $n$ tends to infinity, then we choose a subsequence $\big\{n_{k}\big\}_{k\geq1}$ of $\mathbb{N}$ such that\begin{equation}\label{nk}
	\lim\limits_{k\rightarrow+\infty}
	\frac{-\log\psi(n_{k})}{\log Q_{n_k}}=
	\liminf\limits_{n\rightarrow+\infty}
	\frac{-\log\psi(n)}{\log Q_n}=\alpha,
\end{equation}
furthermore, we require the sequence $\big\{n_{k}\big\}_{k\geq1}$ to be as sparse as possible such that 
\begin{equation}\label{psidelta}
	\frac{\log\psi(n_{k})}{\log\delta_{k}}
	\rightarrow\infty~{\rm{as}}~k\rightarrow
	\infty~{\rm{or~for~example}}~\delta_{k}
	\ge\psi(n_{k})^{1/k}.\end{equation}
We choose a sparse subsequence of $\{n_k\}$ (for simplicity, we still denote by $\{n_k\}$) inductively and then define $\{\delta_k, t_k, N_k, p_k, r_k\}$ as follows. Choose an integer $n_1$ sufficiently large satisfying
\begin{equation}\label{qQ}
	% &\text{(i)}~ \alpha \le \frac{-\log \psi(n_1)}{\log Q_{n_1}} < \alpha+1; & \notag \\
	\frac{\log q_m}{\log Q_m} < \eta ~\text{holds for any}~ m \ge n_1.
	%&\text{(iii)}~\text{we can choose}~ 0<\delta_1<1/2~ \text{such that}~ \delta_1\ge \psi(n_1)^{1/2}& \notag.
\end{equation}
Write $n_1 = p_1M + r_1$ for some integers $p_1\in  \mathbb{N}$ and $0 \le r_1 < M$. Then define integers $t_1$ and $N_1$ such that
$$\frac{Q_{n_1}}{Q_{n_1+t_1+1}}\le\psi(n_1)<\frac{Q_{n_1}}{Q_{n_1+t_1}}$$
and
$$\frac{Q_{n_1}}{Q_{n_1+N_1}}\le\Big|\big((1-\delta_{1})\psi(n_{1}), \psi(n_{1})\big)\Big|<\frac{Q_{n_1}}{Q_{n_1+N_1-1}}.$$

Assume that $n_{k-1}, \delta_{k-1}, t_{k-1}, N_{k-1}, p_{k-1}, r_{k-1}$ have been defined. Choose an integer $n_k$ large enough satisfying
\begin{equation}\label{nkxishu} 
	Q_{n_{k-1}+ N_{k-1}}\le Q_{n_k}^{\epsilon/4}.
\end{equation}

Write $n_k =n_{k-1}+ N_{k-1}+ p_kM + r_k$ for some integers $p_k\in  \mathbb{N}$ and $0 \le r_k < M$. Then define integers $t_k$ and $N_k$ such that
\begin{equation}\label{tk}
		\frac{Q_{n_k}}{Q_{n_k+t_k+1}}\le\psi(n_k)<\frac{Q_{n_k}}{Q_{n_k+t_k}}
\end{equation}and
\begin{equation}\label{Nk}
	\frac{Q_{n_k}}{Q_{n_k+N_k}}\le\Big|\big((1-\delta_{k})\psi(n_{k}), \psi(n_{k})\big)\Big|<\frac{Q_{n_k}}{Q_{n_k+N_k-1}}.\end{equation}
Hence, $N_k>t_k,$ and $N_k-t_k= o(n_k).$ 

Recall that for any $n$-th cylinder $I_n,$ $|I_n|=1/Q_n,$ then there exists a cylinder of order $N_k$ contained in $\big((1-\delta_{k})\psi(n_{k}), \psi(n_{k})\big),$ denoted by $I_{N_{k}}.$ Since  every element in $I_{N_{k}}$ is smaller than $\psi(n_{k}),$ then the corresponding word begins with at least $t_{k}$ zeros. We write this word in the form
	\[
	(\underbrace{0,\ldots,0}_{t_k\text{ zeros}},a_k,\sigma_{k,1},\ldots,\sigma_{k,N_k-t_k-1}),
	\]
	where
	\[
	a_k\in\{1,\ldots,q_{n_k+t_k+1}-1\}\quad
	\text{and}\quad 
	\sigma_k=(\sigma_{k,1},\ldots,\sigma_{k,N_k-t_k-1}).
	\]
	Therefore, for any point $x$ whose digit sequence satisfies $$\varepsilon_{n_k+t_k+1}=a_k, \varepsilon_{n_{k}+t_{k}+2}=\sigma_{k,1},\ldots, \varepsilon_{n_{k}+N_{k}}=\sigma_{k,N_k-t_k-1},$$ 
	we have
$$Q_{n_k}\cdot\big(x-\omega_{n_k}(x)\big)\in\big((1-\delta_k)\psi(n_k), \psi(n_k)\big).$$

Now we give further restrictions on the digits $\varepsilon_k$ to construct a Cantor subset of $E_c(\psi)$. Define $\mathcal{C}_{\infty}$ as the collection
of points whose digit sequence $(\varepsilon_1, \varepsilon_2, \ldots)$ satisfying
\begin{subequations}\label{eq6}
	\begin{empheq}[left=\empheqlbrace]{align}
		&\varepsilon_{n_{k}}\in\{1,\ldots q_{n_{k}}-1\}\qquad\qquad\qquad\qquad\qquad\qquad\qquad\quad\, \; k\geq1;\label{eq6a}\\
		&\varepsilon_{n_{k}+1}=\cdots=\varepsilon_{n_{k}+t_{k}}=0\;\;\qquad\quad\qquad\qquad\quad\quad\quad\quad \ \;\;\,\quad k\geq1;\label{eq6b}\\
		&\varepsilon_{n_{k}+t_{k}+1}={a_k} \;\;\;\,\qquad\quad\quad\quad\qquad\quad\qquad\qquad\qquad\qquad\quad\  k\geq1;\label{eq6c}\\
		&(\varepsilon_{n_{k}+t_{k}+2},\ldots \varepsilon_{n_{k}+N_{k}})=\sigma_k\qquad\qquad\quad\qquad\quad\quad\qquad\quad\,\; \; k\geq1;\label{eq6d} \\
		& \varepsilon_{n_{k}+N_{k}+1}\in\{1,\ldots {q_{n_{k}+N_{k}+1}}-1\}\qquad \quad\qquad\,\quad\qquad\quad\ \ \,\, k\geq1;\label{eq6e} \\
		&\varepsilon_{n_{k}+N_{k}+(j-1)M+1}\in\{1,\ldots, q_{n_{k}+N_{k}+(j-1)M+1}-1\}\qquad \ 1\leq j\leq p_{k+1};\label{eq6f} \\
		&\varepsilon_{n_{k}+N_{k}+p_{k+1}M+1}=\cdots=\varepsilon_{n_{k+1}-1}=0\qquad\qquad\quad\qquad\;\;\,\,\,\   k\geq1;\label{eq6g} \\
		&\varepsilon_{i}\in\{0,\ldots,{q_{i}-1}\}\;\qquad\qquad\qquad \qquad\qquad\qquad\qquad\quad\; \text{otherwise},
	\end{empheq}
\end{subequations}
where $n_{0}=t_{0}=0.$ We also define for each $n\ge 1$, $$
\mathcal{C}_n=\Big\{I(\varepsilon_1,\cdots, \varepsilon_n): I(\varepsilon_1,\cdots, \varepsilon_n)\cap \mathcal{C}_{\infty}\ne \emptyset\Big\}.$$

We are now in the place to check that $\mathcal{C}_{\infty}\subset E_c(\psi).$ Let $x\in\mathcal{C}_{\infty}$. As before, the digit sequence of $x$ is denoted by $\big\{\varepsilon_{i}(x)\big\}_{i\ge1}$. 
%We consider the quantity $Q_{n}\big(x-\omega_{n}(x)\big)$ for all $n\geq n_{1}.$} 
\begin{lemma}
$\mathcal{C}_{\infty}\subset E_c(\psi).$
\end{lemma}
\begin{proof}
Let $k$ be the integer such that $n_{k}\leq n<n_{k+1}.$ When $n=n_k$, by the construction of  $\mathcal{C}_{\infty}$, we have
$$Q_{n_k}\cdot\big(x-\omega_{n_k}(x)\big)\in\big((1-\delta_k)\psi(n_k), \psi(n_k)\big).$$
Therefore, 
$$x-\omega_{n_k}(x)<\frac{\psi(n_k)}{Q_{n_k}}~\text{holds for infinitely many}~n_k.$$

Now we need to prove that there exists an integer $K$ such that for any $k\ge K$ and $n_k\le n<n_{k+1},$ one has 
$$Q_{n}\big(x-\omega_{n}(x)\big)\ge (1-\delta_k)\psi(n).$$ 

(I) When $n=n_{k},$ by the construction of  $\mathcal{C}_{\infty}$, we have
$$Q_{n_k}\cdot\big(x-\omega_{n_k}(x)\big)\in\big((1-\delta_k)\psi(n_k), \psi(n_k)\big).$$
Therefore, 
$$Q_{n_k}\cdot\big(x-\omega_{n_k}(x)\big)\ge (1-\delta_k)\psi(n_k).$$

(II) When $n_{k}<n\leq n_{k}+t_{k}.$ Since $\psi$ is a non-increasing function, we have
\begin{align*}
	Q_{n}\big(x-\omega_{n}(x)\big)&\overset{\textrm{by}\ (\ref{eq6b})}{=}\frac{Q_{n}}{Q_{n_k}} Q_{n_k}\cdot\big(x-\omega_{n_k}(x)\big)\\
	&\ge Q_{n_k}\cdot\big(x-\omega_{n_k}(x)\big)\\&\ge (1-\delta_k)\psi(n_k)\\&\ge (1-\delta_k)\psi(n).
\end{align*}

(III) When $n_{k}+t_k<n\leq n_{k}+N_{k}.$ Recall that $\varepsilon_{n_k+N_k+1}\neq 0,$ we have
\begin{align*}
	Q_{n}\big(x-\omega_{n}(x)\big)\ge\frac{Q_n}{Q_{n_k+N_k+1}}\ge\frac{Q_{n_k+t_k+1}}{Q_{n_k+N_k+1}}
	=\frac{Q_{n_k+t_k+1}}{Q_{n_k+N_k}\cdot q_{n_k+N_k+1}}.
\end{align*}
By (\ref{tk}) and (\ref{Nk}), we have
\begin{align*} \frac{Q_{n_k+t_k+1}}{Q_{n_k+N_k}}&=\frac{Q_{n_k+t_k+1}/Q_{n_k}}{Q_{n_k+N_k}/Q_{n_k}}\ge\frac{\psi(n_k)^{-1}}{(\delta_k \psi(n_k))^{-1}\cdot q_{n_k+N_k}}=\frac{\delta_k}{q_{n_k+N_k}}.
\end{align*}
Hence, 
$$Q_{n}\big(x-\omega_{n}(x)\big)\ge\frac{\delta_k}{q_{n_k+N_k}}\cdot\frac{1}{q_{n_k+N_k+1}}
\overset{\textrm{by}\ (\ref{qQ})}{\ge}\delta_k\cdot Q^{-2\eta}_{n_k+N_k+1}.$$
By (\ref{Nk}), we have
$$\frac{Q_{n_k}}{Q_{n_k+N_k-1}}>\delta_k\psi(n_k)\Rightarrow Q_{n_k+N_k}\le \frac{Q_{n_k}}{\delta_k\psi(n_k)}\cdot q_{n_k+N_k}.$$
Then,
$$Q_{n_k+N_k+1}=Q_{n_k+N_k}\cdot q_{n_k+N_k+1}\le \frac{Q_{n_k}}{\delta_k\psi(n_k)}\cdot q_{n_k+N_k}\cdot q_{n_k+N_k+1}.$$
Since $q_{n_k+N_k}\cdot q_{n_k+N_k+1}\le Q^{2\eta}_{n_k+N_k+1},$ one has
$$Q_{n_k+N_k+1}^{1-2\eta}\le \frac{Q_{n_k}}{\delta_k\psi(n_k)}\Rightarrow
Q_{n_k+N_k+1}\le \Big(\frac{Q_{n_k}}{\delta_k\psi(n_k)}\Big)^{\frac{1}{1-2\eta}}.$$
Then 
$$Q_{n}\big(x-\omega_{n}(x)\big)\ge\delta_k\cdot \Big(\frac{\delta_k\psi(n_k)}{Q_{n_k}}\Big)^{\frac{2\eta}{1-2\eta}}= \delta_k^{1+\frac{2\eta}{1-2\eta}}\cdot\Big(\frac{\psi(n_k)}{Q_{n_k}}\Big)^{\frac{2\eta}{1-2\eta}}.$$
Let
\[
\alpha_k
:=
\frac{-\log\psi(n_k)}
{\log Q_{n_k}}\longrightarrow\alpha.
\]
Consequently,
\[
\frac{\psi(n_k)}{Q_{n_k}}
=
\psi(n_k)^{1+1/\alpha_k}.
\]
Recall that $\delta_k\ge \psi(n_k)^{1/k},$ then
$$Q_{n}\big(x-\omega_{n}(x)\big)\ge \psi(n_k)^{\frac{1}{k}(1+\frac{2\eta}{1-2\eta})+\frac{2\eta(1+\alpha_k)}{\alpha_k(1-2\eta)}}
=\psi(n_k)^{\beta_k(\eta)},$$
where $\beta_k(\eta):=\frac{1}{k}(1+\frac{2\eta}{1-2\eta})+\frac{2\eta(1+\alpha_k)}{\alpha_k(1-2\eta)}.$ By the definition of $\eta$ and the fact that $1/k\rightarrow 0$, there exists an integer $K_0$ such that for all $k\ge K_0$, we have  $\beta_k(\eta)<1$. Since $\psi(n)\rightarrow 0$, there exists an integer $K_1$ such that
when $k\ge \max\{K_0, K_1\}$, one has 
$$\psi(n_k)^{\beta_k(\eta)-1}\ge2>1-\delta_k,$$ 
thus,
\begin{equation*}
	Q_{n}\big(x-\omega_{n}(x)\big)\ge \psi(n_k)^{\beta_k(\eta)}>(1-\delta_k)\psi(n).
\end{equation*}

(IV) If there exists an integer $1\leq j\leq p_{k+1}-1$ such that \[n_{k}+N_{k}+(j-1)M< n\leq n_{k}+N_{k}+jM.\]
Recall that $\epsilon_{n_{k}+N_{k}+jM+1}(x)\neq0$ for any $1\leq j\leq p_{k+1}-1,$ then
\begin{align*}
	Q_{n}\big(x-\omega_{n}(x)\big)=Q_n
	\big(\frac{\varepsilon_{n+1}}{Q_{n+1}}+\cdots
	\frac{\varepsilon_{n_{k}+N_{k}+jM+1}}
	{Q_{n_{k}+N_{k}+jM+1}}+\cdots\big)\ge\frac{Q_n}
	{Q_{n_{k}+N_{k}+jM+1}}.
	%\textcolor{red}{\ge c_4Q_{n}^{-\eta M} }
\end{align*}
By \ref{qQ}, one has
$$\frac{Q_{n_{k}+N_{k}+jM+1}}{Q_n}=\prod\limits_{i=n+1}^{n_{k}+N_{k}+jM+1} q_i\le Q_{n_{k}+N_{k}+jM+1}^{M\eta},$$
then
$$Q_n\ge (Q_{n_{k}+N_{k}+jM+1})^{1-M\eta}.$$
Since $\eta M< \frac{\alpha}{4+\alpha}$, we have 
$$\frac{Q_n}{Q_{n_{k}+N_{k}+jM+1}} \geq Q_{n_{k}+N_{k}+jM+1}^{-M\eta} \geq Q_n^\frac{-\eta M}{1-\eta M}\ge Q_n^{-\alpha/4}.$$
By (\ref{nk}), there exists an integer $K_2$ such that when $k\ge K_2$, one has
$$\psi(n)\le Q_n^{-\alpha/2}.$$
Then,
\begin{align*}
	Q_{n}\big(x-\omega_{n}(x)\big)&\ge  Q_n^{-\alpha/4}\\
	&= Q_n^{\alpha/4}\cdot Q_n^{-\alpha/2}\\&\ge  Q_n^{\alpha/4}\psi(n).\end{align*}
Since $n>n_k\rightarrow\infty, Q_n^{\alpha/4}\rightarrow\infty,$ then there exists an integer $K_3$ such that when $k\ge K_3,$ we have
$$ Q_n^{\alpha/4}\ge 2>1-\delta_k,$$
thus,
$$Q_{n}\big(x-\omega_{n}(x)\big)> (1-\delta_k)\psi(n)$$
holds for any $k\ge \max\{K_2,K_3\}.$

(V) Suppose $n_{k}+N_{k}+(p_{k+1}-1)M< n<n_{k+1},$ recall that $\epsilon_{n_{k+1}}(x)\neq0,$ we have 
$$Q_{n}\big(x-\omega_{n}(x)\big)\ge
\frac{Q_n}{Q_{n_{k+1}}}.$$ 
Notice that
\[
n_{k+1}
=
n_k+N_k+p_{k+1}M+r_{k+1},
\qquad
0\le r_{k+1}<M,
\]
so in the present case
\[
n_{k+1}-n<2M.
\]
Therefore, by \eqref{qQ},
\[
\frac{Q_{n_{k+1}}}{Q_n}
\le
Q_{n_{k+1}}^{2M\eta}.
\]
It follows that
\[
Q_{n_{k+1}}^{1-2M\eta}
\le Q_n
\]
and hence
\[
\frac{Q_n}{Q_{n_{k+1}}}
\ge
Q_n^{-\frac{2M\eta}{1-2M\eta}}.
\]
By \eqref{eta-M-alpha}, we have $\frac{2M\eta}{1-2M\eta}<\frac{\alpha}{4}$.
Thus
\[
Q_{n}\big(x-\omega_{n}(x)\big)\ge Q_n^{-\alpha/4}.
\]
The same argument as in Case (IV) now yields
\[
Q_{n}\big(x-\omega_{n}(x)\big)
>
(1-\delta_k)\psi(n)
\]
holds for any $k\ge \max\{K_2,K_3\}.$

Therefore, for any $x\in\mathcal{C}_{\infty},$ combining (I)-(V) together, we have that the inequality $Q_{n}\big(x-\omega_{n}(x)\big)
<(1-\delta_{k})\psi(n)$ has no solution in integers $n$ for all $k\ge K=\max\{K_0,K_1,K_2,K_3\}$ and $n_{k}\leq n<n_{k+1}.$ This implies that $\mathcal{C}_{\infty}\subset E_{c}(\psi).$
\end{proof}

\subsubsection{Supporting measure}
We start with a notation: let $I(\varepsilon_1,\cdots, \varepsilon_{n-1})$ be an element in $\mathcal{C}_{n-1}$ with $n\ge 2$, denote by $$
\mathcal{C}_n(\varepsilon_1,\cdots, \varepsilon_{n-1})=\Big\{I(\varepsilon_1,\cdots, \varepsilon_n):I(\varepsilon_1,\cdots, \varepsilon_n)\in \mathcal{C}_n\Big\}
$$ i.e. the offspring of $I(\varepsilon_1,\cdots, \varepsilon_{n-1})$ in $\mathcal{C}_n$.

Now we distribute a probability measure $\mu$ on $\mathcal{C}_{\infty}.$ Firstly, define $\mu\big([0, 1]\big)=1.$ By Kolmogorov extension theorem, it suffices to define its value on cylinders since they form a semi-algebra. 

Let $I(\varepsilon_1)$ be a cylinder in $\mathcal{C}_{1}$. By the construction of $\mathcal{C}_{\infty},$
% or the digits requirement (\ref{eq6}), 
there are $q_1$ cylinders in $\mathcal{C}_1$. Thus  for each $I(\varepsilon_1)$ in $\mathcal{C}_{1}$, define $$
\mu(I(\varepsilon_1))=q_1^{-1}.
$$

Then we define the measure on all cylinders in $\mathcal{C}_n$ inductively. Assume the measure of cylinders in $\mathcal{C}_{n-1}$ has been defined. 
Let $I(\varepsilon_1,\cdots, x_n)$ be a cylinder in $\mathcal{C}_n$. 
Define \begin{equation}
\mu\Big(I(\varepsilon_{1},\ldots,\varepsilon_{n})\Big)=
\frac{1}{\#\mathcal{C}_{n}\big(\varepsilon_{1},\ldots,\varepsilon_{n-1}\big)}\cdot
\mu\Big(I(\varepsilon_{1},\ldots,\varepsilon_{n-1})\Big),\nonumber
\end{equation}
where the symbol ``$\#$'' denotes the cardinality of a set. That is, the measure of a mother basic cylinder is evenly distributed among its offspring. 

We have the following expressions for its measure.

(I) When $n=n_{k},$ by the construction of $\mathcal{C}_{\infty}$ or the requirement on digits (\ref{eq6}), we have
\begin{equation}\label{munk}
\begin{aligned}
	\mu\big(I(\varepsilon_{1},\ldots,\varepsilon_{n_k})\big)&=
	\mu\big(I(\varepsilon_{1},\ldots,\varepsilon_{n_k-1})\big)\cdot\frac{1}{q_{n_k}-1}
	\\&=\mu\big(I(\varepsilon_{1},\ldots,\varepsilon_{n_{k-1}+N_{k-1}+p_{k}M})\big)\cdot \frac{1}{q_{n_k}-1}\\
	&=\mu\big(I(\varepsilon_{1},\ldots,\varepsilon_{n_{k-1}+N_{k-1}})\big)
	\cdot\Big(\prod_{j=1}^{p_k}F_{k-1,j}^{-1}\big)\cdot\frac{1}{q_{n_k}-1}\\
	&=\mu\big(I(\varepsilon_{1},\ldots,\varepsilon_{n_{k-1}})\big)
	\cdot\Big(\prod_{j=1}^{p_k}F_{k-1,j}^{-1}\Big)\cdot\frac{1}{q_{n_k}-1},
\end{aligned}
\end{equation}
where 
\begin{equation*}
\begin{aligned}
	F_{k-1,j}
	&=(q_{n_{k-1}+N_{k-1}+(j-1)M+1}-1)\cdot\prod\limits_{l=2}^{M}q_{n_{k-1}+N_{k-1}+(j-1)M+l}\\
	&\ge 2^{-1}\prod_{l=1}^Mq_{n_{k-1}+N_{k-1}+(j-1)M+l}.
\end{aligned}
\end{equation*}

(II) When $n_{k}<n\leq n_{k}+N_k,$ we have
$$\mu\big(I(\varepsilon_{1},\ldots,\varepsilon_{n})\big)
=\mu\big(I(\varepsilon_{1},\ldots,\varepsilon_{n_{k}})\big).$$

(III) When $n=n_{k}+N_{k}+1,$ we have
\[
\begin{aligned}
\mu\big(I(\varepsilon_1,\ldots,\varepsilon_n)\big)
&=
\frac1{q_{n_k+N_k+1}-1}
\mu\big(I(\varepsilon_1,\ldots,\varepsilon_{n_k+N_k})\big)
\\
&\le
\frac2{q_{n_k+N_k+1}}
\mu\big(I(\varepsilon_1,\ldots,\varepsilon_{n_k+N_k})\big).
\end{aligned}
\]

(IV) When $n=n_{k}+N_{k}+pM$ for some $1\leq p\leq p_{k+1},$ we have
$$\mu\big(I(\varepsilon_{1},\ldots,\varepsilon_{n})\big)
=\mu\big(I(\varepsilon_{1},\ldots,\varepsilon_{n_{k}+N_{k}})\big)\cdot\prod_{j=1}^{p}F_{k,j}^{-1}.$$

(V) When $n_{k}+N_{k}+(p-1)M<n\leq n_{k}+N_{k}+pM-1,$ for some $1\leq p\leq p_{k+1},$ we have
$$\mu\big(I(\varepsilon_{1},\ldots,\varepsilon_{n_{k}+N_{k}+pM})\big)\le
\mu\big(I(\varepsilon_{1},\ldots,\varepsilon_{n})\big)\le \mu\big(I(\varepsilon_{1},\ldots,\varepsilon_{n_{k}+N_{k}+(p-1)M})\big).$$

(VI) When $n=n_{k}+N_{k}+p_{k+1}M<n<n_{k+1},$ we have
$$\mu\big(I(\varepsilon_{1},\ldots,\varepsilon_{n})\big)=
\mu\big(I(\varepsilon_{1},\ldots,\varepsilon_{n_{k}+N_{k}+p_{k+1}M})\big).$$

\subsubsection{Estimation on the $\mu$-measure of cylinders}
We are now in a place to estimate the H{\"o}lder exponent of the measure $\mu$ to make use of the mass distribution principle (Lemma \ref{3.3}). So, we need to compare the length of a cylinder with its measure. By (\ref{nk}), for any given arbitrarily small $0<\epsilon<1$, there exists an integer $K$, such that for any $k>K$, one has $$\alpha-\epsilon\le\frac{-\log \psi(n_k)}{\log Q_{n_k}}\le\alpha+\epsilon.$$ Hence, 
\begin{equation}\label{cedunkreq}
Q_{n_k}^{-(\alpha+\epsilon)}\le \psi(n_k)\le Q_{n_k}^{-(\alpha-\epsilon)}.
\end{equation}
Furthermore, by the sparsity of $\{n_k\}$, we have 
\begin{equation}\label{cedudeltakreq}
\delta_k^{-1}\le Q_{n_k}^{\epsilon}.
\end{equation}
Let $k>K$ be the integer such that $n_{k}\leq n<n_{k+1}.$ We first give the following proposition.

\begin{proposition}\label{pro}
By the construction of $\mathcal{C}_{\infty},$ we have
$$\prod_{j=n_{k-1}+N_{k-1}}^{n_k-M}q_j\cdot q_{n_k}\ge Q_{n_k}^{1-\epsilon/2}.$$
\end{proposition}
\begin{proof}
Write
$$\prod_{j=n_{k-1}+N_{k-1}}^{n_k-M}q_j\cdot q_{n_k}=\frac{Q_{n_k-M}}{Q_{n_{k-1}+N_{k-1}-1}}\cdot q_{n_k}.$$
By (\ref{nkxishu}) and (\ref{qQ}), we have
$$
Q_{n_{k-1}+N_{k-1}-1}\le Q_{n_k}^{\epsilon/4}$$
and
$$
\prod_{j=n_{k}-M+1}^{n_k}q_j\le Q_{n_k}^{M\eta}\le Q_{n_k}^{\epsilon/4},
$$
where the last inequality is due to the fact that $\eta<\frac{\epsilon}{4M}.$
Hence,
$$Q_{n_k-M}=\frac{Q_{n_k}}{\prod\limits_{j=n_{k}-M+1}^{n_k}q_j}\ge Q_{n_k}^{1-\epsilon/4},$$
then
$$\prod_{j=n_{k-1}+N_{k-1}}^{n_k-M}q_j\cdot q_{n_k}=\frac{Q_{n_k-M}}{Q_{n_{k-1}+N_{k-1}-1}}\cdot q_{n_k}\ge Q_{n_k}^{1-\epsilon/4}\cdot Q_{n_k}^{-\epsilon/4}=Q_{n_k}^{1-\epsilon/2}.$$
\end{proof}

(I) When $n=n_k$. Recall (\ref{munk}), we have
\begin{equation}\label{honk}
\begin{aligned}
	\mu\big(I(\varepsilon_1,\cdots, \varepsilon_{n_{k}})\big) &\le \prod_{j=1}^{k} \frac{2}{q_{n_j} } \cdot 2^{p_j} \cdot \frac{Q_{n_{j-1}+N_{j-1}}}{Q_{n_j-r_j}}\\
	&\le \Big(\frac{1}{q_1\cdots q_{n_k}}\Big)^{1-\epsilon} = |I(\varepsilon_1,\cdots, \varepsilon_{n_{k}})|^{1-\epsilon},\end{aligned}
	\end{equation}
	where the last inequality is from (\ref{request}), (\ref{nkxishu}) and Proposition \ref{pro}.

	(II) When $n_k<n\le n_k+t_k$, we have
	$$\mu\big(I(\varepsilon_{1},\ldots,\varepsilon_{n})\big)
	=\mu\big(I(\varepsilon_{1},\ldots,\varepsilon_{n_k})\big)
	\le|I(\varepsilon_{1},\ldots,\varepsilon_{n_k})|^{1-\epsilon}.
	$$
	Note that
	\begin{equation*}
\begin{aligned}
	|I(\varepsilon_{1},\ldots,\varepsilon_{n})|&=\frac{1}{Q_n}
	=\frac{1}{Q_{n_k}}\cdot \frac{Q_{n_k}}{Q_n}\\
	&\ge\frac{1}{Q_{n_k}}\cdot \frac{Q_{n_k}}{Q_{n_k+t_k}}\\
	&\overset{{\text{by}}\ (\ref{tk})}>\frac{1}{Q_{n_k}}\cdot
	\psi(n_k)\\
	&=\psi(n_k)\cdot|I(\varepsilon_{1},\ldots,\varepsilon_{n_k})|. 
\end{aligned}
\end{equation*}
Since $\psi(n_k)\ge Q_{n_k}^{-(\alpha+\epsilon)}$, we have		\begin{equation*}
\begin{aligned}
	|I(\varepsilon_{1},\ldots,\varepsilon_{n_k})|&\le
	\frac{1}{\psi(n_k)}\cdot|I(\varepsilon_{1},\ldots,\varepsilon_{n})|\\
	&\le Q_{n_k}^{\alpha+\epsilon}\cdot|I(\varepsilon_{1},\ldots,\varepsilon_{n})|\\
	&=|I(\varepsilon_{1},\ldots,\varepsilon_{n_k})|^{-(\alpha+\epsilon)}
	\cdot|I(\varepsilon_{1},\ldots,\varepsilon_{n})|,
\end{aligned}
\end{equation*}
%hence
%$$|I(\varepsilon_{1},\ldots,\varepsilon_{n_k})|\le
%|I(\varepsilon_{1},\ldots,\varepsilon_{n})|^{\frac{1}{1+\alpha+\frac{1}{k}+\epsilon}}.
%$$
hence,
\begin{equation*}
\begin{aligned}
	\mu\big(I(\varepsilon_{1},\ldots,\varepsilon_{n})\big)&\le
	|I(\varepsilon_{1},\ldots,\varepsilon_{n_k})|^{1-\epsilon}
	\\&\le|I(\varepsilon_{1},\ldots,\varepsilon_{n})|^{\frac{1-\epsilon}{1+\alpha+\epsilon}}.
\end{aligned}
\end{equation*}

(III) When $n_{k}+t_{k}+1\le n< n_{k}+N_{k},$ we have
\begin{equation*}
\begin{aligned}
	\mu\big(I(\varepsilon_{1},\ldots,\varepsilon_{n})\big)&=
	\mu\big(I(\varepsilon_{1},\ldots,\varepsilon_{n_{k}})\big)\\
	%&\le\frac{2}{q_{n_{k}+t_{k}+1}}\cdot\mu\big(I(\varepsilon_{1},\ldots,\varepsilon_{n_{k}})\big)\\
	&\overset{{\text{by}}\ (\ref{honk})}{\le}|I(\varepsilon_1,\cdots, \varepsilon_{n_k})|^{1-\epsilon}.
	%&=2\Big(|I(\varepsilon_1,\cdots, \varepsilon_{n_k+t_k+1})| \cdot q_{n_k+t_k+1}\Big)^{\frac{1-\epsilon}{1+\alpha}}\cdot \frac{1}{q_{n_{k}+t_{k}+1}}\\
	%&=2|I(\varepsilon_1,\cdots, \varepsilon_{n_k+t_k+1})|^{\frac{1-\epsilon}{1+\alpha}} \cdot
	%q_{n_k+t_k+1}^{-\frac{\alpha+\epsilon}{1+\alpha}}\\
	%&\le 2|I(\varepsilon_1,\cdots, \varepsilon_{n_k+t_k+1})|^{\frac{1-\epsilon}{1+\alpha}}.
\end{aligned}
\end{equation*}
Recall (\ref{tk}) and (\ref{Nk}), we have
\begin{equation}\label{iiicd}
\begin{aligned}
	|I(\varepsilon_1,\cdots, \varepsilon_n)|&=\frac{1}{Q_n}=\frac{1}{Q_{n_k}}
	\cdot\frac{Q_{n_k}}{Q_n}\\&\ge
	\frac{1}{Q_{n_k}}\cdot\frac{Q_{n_k}}{Q_{n_k+N_k-1}}
	\\&\ge \delta_k\frac{\psi(n_k)}{Q_{n_k}}=\delta_k\psi(n_k)|I(\varepsilon_1,\cdots, \varepsilon_{n_k})|
\end{aligned}
\end{equation}
By (\ref{cedunkreq}) and (\ref{cedudeltakreq}), we have
$$\delta_k\psi(n_k)\ge Q_{n_k}^{-\epsilon}\cdot Q_{n_k}^{-(\alpha+\epsilon)}=Q_{n_k}^{-(\alpha+2\epsilon)}.$$
Hence
$$|I(\varepsilon_1,\cdots, \varepsilon_{n_k})|\le |I(\varepsilon_1,\cdots, \varepsilon_n)|^{\frac{1}{1+\alpha+2\epsilon}}.$$
Thus,
\begin{equation*}
\begin{aligned}
	\mu\big(I(\varepsilon_1,\cdots, \varepsilon_n)\big)&=\mu\big(I(\varepsilon_{1},\ldots,\varepsilon_{n_{k}})\big)\\
	%&<|I(\varepsilon_1,\cdots, \varepsilon_{n_k+t_k+1})|^{\frac{1-\epsilon}{1+\alpha+2\epsilon}}\\
	&\le |I(\varepsilon_1,\cdots, \varepsilon_{n_k})|^{1-\epsilon}\\
	&\le |I(\varepsilon_1,\cdots, \varepsilon_{n})|^{\frac{1-\epsilon}{1+\alpha+2\epsilon}}.
\end{aligned}
\end{equation*}

(IV) When $n=n_k+N_k,$ for convenience, we denote $I(\varepsilon_1,\cdots, \varepsilon_n)$ by
$J_k$, then
\begin{equation*}
\begin{aligned}
	\mu(J_k)&=\mu(J_{k-1})\cdot\frac{1}{q_{n_k}-1}\cdot \prod\limits_{j=1}^{p_k}
	F_{k-1,j}^{-1}\\
	&\le \mu(J_{k-1})\cdot\frac{2}{q_{n_k}}\cdot \prod\limits_{j=1}^{p_k}
	F_{k-1,j}^{-1}
\end{aligned}
\end{equation*}
where
$$F_{k-1,j}\ge 2^{-1}\prod_{l=1}^{M} q_{n_{k-1}+N_{k-1}+(j-1)M+l}.$$
One has
\begin{equation*}
\begin{aligned}
	\prod\limits_{j=1}^{p_k}F_{k-1,j}^{-1}&\le \prod\limits_{j=1}^{p_k}\Big(2^{-1}\prod_{l=1}^{M} q_{n_{k-1}+N_{k-1}+(j-1)M+l}\Big)^{-1}\\
	&=2^{p_k}\cdot \prod\limits_{j=1}^{p_k}\frac{1}{\prod\limits_{l=1}^{M}q_{n_{k-1}+N_{k-1}+(j-1)M+l}}\\
	&=2^{p_k}\cdot \frac{Q_{n_{k-1}+N_{k-1}}}{Q_{n_k-r_k}}.
\end{aligned}
\end{equation*}
Hence
$$
\frac{\mu(J_k)}{\mu(J_{k-1})}\le 2\cdot \frac{2^{p_k}}{q_{n_k}}\cdot \frac{Q_{n_{k-1}+N_{k-1}}}{Q_{n_k-r_k}}.
$$
Note that 
$$\frac{Q_{n_k}}{Q_{n_k-r_k}}=\prod\limits_{m=n_{k}-r_k+1}^{n_k}q_m\overset{{\text{by}}\ (\ref{qQ})}{\le}
Q_{n_k}^{\eta M}\le Q_{n_k}^{\epsilon}
$$
and 
$$\frac{Q_{n_k}}{Q_{n_{k-1}+N_{k-1}}}\ge\frac{Q_{n_k-r_k}}{Q_{n_{k-1}+N_{k-1}}}=
\prod\limits_{m=n_{k-1}+N_{k-1}+1}^{n_k-r_k}q_m\ge 2^{p_k M},$$
we have 
$$Q_{n_k-r_k}\ge Q_{n_k}^{1-\epsilon}~\text{and}~2^{p_k}\le \Big(\frac{Q_{n_k}}{Q_{n_{k-1}+N_{k-1}}}\Big)^{\frac{1}{M}}\le Q_{n_k}^{\epsilon}.$$
Then
$$\frac{\mu(J_k)}{\mu(J_{k-1})}\le 2\cdot \frac{2^{p_k}}{q_{n_k}}\cdot \frac{Q_{n_{k-1}+N_{k-1}}}{Q_{n_k-r_k}}\le Q_{n_k}^{-(1-2\epsilon)}\cdot Q_{n_{k-1}+N_{k-1}}.$$
Now we deal with $|J_k|$, we have
$$\frac{|J_k|}{|J_{k-1}|}=\frac{Q_{n_{k-1}+N_{k-1}}}{Q_{n_{k}+N_k}},$$
by the definition of $N_k$ and (\ref{psidelta}) , we have
\begin{equation*}
\begin{aligned}
	Q_{n_k+N_k}&=Q_{n_k+N_k-1}\cdot q_{n_k+N_k}\le \frac{Q_{n_k}}{\delta_k \psi(n_k)}\cdot q_{n_k+N_k}\\
	&<Q_{n_k}\cdot Q_{n_k}^{\epsilon}\cdot Q_{n_k}^{\alpha+\epsilon}\cdot Q_{n_k}^{\epsilon}=Q_{n_k}^{1+\alpha+3\epsilon},
\end{aligned}
\end{equation*}
then
$$\frac{|J_k|}{|J_{k-1}|}>\frac{Q_{n_{k-1}+N_{k-1}}}{Q_{n_k}^{1+\alpha+3\epsilon}}.$$
Therefore, we have
\begin{equation*}
\begin{aligned}
	\log\frac{\mu(J_k)}{\mu(J_{k-1})}&\le -(1-2\epsilon)\log Q_{n_k}+\log Q_{n_{k-1}+N_{k-1}}\\ 
	&\overset{\text{by }\ref{nkxishu}}{\le}-(1-3\epsilon)\log Q_{n_k}
\end{aligned}
\end{equation*}
and 
\begin{equation*}
\begin{aligned}
	\log\frac{|J_k|}{|J_{k-1}|}\ge \log Q_{n_{k-1}+N_{k-1}}-(1+\alpha+3\epsilon)\log Q_{n_k}\ge -(1+\alpha+3\epsilon)\log Q_{n_k}.
\end{aligned}
\end{equation*}
Thus,
$$\frac{-\log\frac{\mu(J_k)}{\mu(J_{k-1})}}{-\log\frac{|J_k|}{|J_{k-1}|}}\ge \frac{1-3\epsilon}{1+\alpha+3\epsilon}.$$
By applying the above inequality inductively for $i=2,\ldots,k$, we obtain
$$\mu(J_k)\le C|J_k|^{\frac{1-3\epsilon}{1+\alpha+3\epsilon}}~\text{for some constant}~C.$$

%(V) When $n=n_k+N_k+1,$ we have 
%	\begin{equation*}
%			\begin{aligned}
	%				\mu\big(I(\varepsilon_{1},\ldots,\varepsilon_{n})\big)&=
	%				\frac{1}{q_{n_k+N_k+1}-1}\cdot\mu(J_k)\\
	%&\le \frac{2}{q_{n_k+N_k+1}}\cdot\mu(J_k)\\				%&\le\frac{2}{q_{n_{k}+t_{k}+1}}\cdot\mu\big(I(\varepsilon_{1},\ldots,\varepsilon_{n_{k}})\big)\\
	%				&\overset{{\text{by}}\ (\text{(IV)})}{\le}\frac{2C}{q_{n_k+N_k+1}}\cdot |J_k|^{\frac{1-3\epsilon}{1+\alpha+3\epsilon}}
	%\end{aligned}
	%\end{equation*}
	%and
	%$$|I(\varepsilon_{1},\ldots,\varepsilon_{n})|
	%=\frac{|J_k|}{q_{n_k+N_k+1}}.$$
	%Then,
	%\begin{equation*}
	%			\begin{aligned}
		%\mu\big(I(\varepsilon_{1},\ldots,\varepsilon_{n})\big)
		%&\le \frac{2C}{q_{n_k+N_k+1}}\cdot |J_k|^{\frac{1-3\epsilon}{1+\alpha+3\epsilon}}\\
		%&=\frac{2C}{q_{n_k+N_k+1}}\cdot \Big(q_{n_k+N_k+1}|I(\varepsilon_{1},\ldots,\varepsilon_{n})|\Big)^{\frac{1-3\epsilon}{1+\alpha+3\epsilon}}\\
		%&=2C \cdot q_{n_k+N_k+1}^{\frac{1-3\epsilon}{1+\alpha+3\epsilon}-1}\cdot|I(\varepsilon_{1},\ldots,\varepsilon_{n})|\Big)^{\frac{1-3\epsilon}{1+\alpha+3\epsilon}}.
		%	\end{aligned}
	%\end{equation*}
	%Since $0<\frac{1-3\epsilon}{1+\alpha+3\epsilon}<1$ and $q_{n_k+N_k+1}\ge2$, then $q_{n_k+N_k+1}^{\frac{1-3\epsilon}{1+\alpha+3\epsilon}-1}<1.$
	%Therefore,
	%$$\mu\big(I(\varepsilon_{1},\ldots,\varepsilon_{n})\big)
	%\le 2C \cdot |I(\varepsilon_{1},\ldots,\varepsilon_{n})|\Big)^{\frac{1-3\epsilon}{1+\alpha+3\epsilon}}.$$

	(V) When $n=n_{k}+N_{k}+pM$ for some $1\le p\le p_{k+1},$ we have
	$$\mu\big(I(\varepsilon_1,\cdots, \varepsilon_n)\big)=\mu(J_k)\cdot \prod\limits_{j=1}^{p}F_{k,j}^{-1}.$$
	Recall that 
	$$F_{k,j}\ge\frac{1}{2}\prod\limits_{l=1}^{M}
	q_{n_k+N_k+(j-1)M+l},$$
	then
	\begin{equation*}
		\begin{aligned}
			\prod\limits_{j=1}^{p}F_{k,j}&\ge 2^{-p}\cdot\prod\limits_{j=1}^{p}\prod\limits_{l=1}^{M}
			q_{n_k+N_k+(j-1)M+l}\\
			&=2^{-p}\cdot\frac{Q_{n_k+N_k+pM}}{Q_{n_k+N_k}}\\
			&=2^{-p}\cdot\frac{|J_k|}{|I(\varepsilon_1,\cdots, \varepsilon_n)|}.
		\end{aligned}
	\end{equation*}
	Hence
	$$\mu\big(I(\varepsilon_1,\cdots, \varepsilon_n)\big)\le 2^{p}\cdot\mu(J_k)\cdot\frac{|I(\varepsilon_1,\cdots, \varepsilon_n)|}{|J_k|}.$$
	Since 
	$$2^{pM}\le \prod\limits
	_{m=n_k+N_k+1}^{n_k+N_k+pM}q_m =\frac{Q_{n_k+N_k+pM}}{Q_{n_k+N_k}}=
	\frac{|J_k|}{|I(\varepsilon_1,\cdots, \varepsilon_n)|},$$
	then
	\begin{equation*}
		\begin{aligned}
			\mu\big(I(\varepsilon_1,\cdots, \varepsilon_n)\big)&\le 2^{p}\cdot\mu(J_k)\cdot\frac{|I(\varepsilon_1,\cdots, \varepsilon_n)|}{|J_k|}\\
			&\le \mu(J_k)\cdot\Big(\frac{|J_k|}{|I(\varepsilon_1,\cdots, \varepsilon_n)|}\Big)^{\frac{1}{M}}\cdot \frac{|I(\varepsilon_1,\cdots, \varepsilon_n)|}{|J_k|}\\&\overset{\text{by (V)}}{\le}C|J_k|
			^{\frac{1-3\epsilon}{1+\alpha+3\epsilon}+\frac{1}{M}-1}
			\cdot|I(\varepsilon_1,\cdots, \varepsilon_n)|^{1-\frac{1}{M}}\\
			&\le C|I(\varepsilon_1,\cdots, \varepsilon_n)|^{\frac{1-3\epsilon}{1+\alpha+3\epsilon}},
		\end{aligned}
	\end{equation*}
	where the last inequality is due to that $\frac{1-3\epsilon}{1+\alpha+3\epsilon}+\frac{1}{M}-1<0.$
	
	(VI) When $n_k+N_k+(p-1)M<n<n_k+N_k+pM$ for some $1\le p\le p_{k+1},$ we have
	\begin{equation*}
		\begin{aligned}\mu\big(I(\varepsilon_1,\cdots, \varepsilon_n)\big)&\le\mu\big(I(\varepsilon_1,\cdots, \varepsilon_{n_k+N_k+(p-1)M})\big)\\&\overset{\text{by (VI)}}\le 
			C|I(\varepsilon_1,\cdots, \varepsilon_{n_k+N_k+(p-1)M})|^{\frac{
					1-3\epsilon}{1+\alpha+3\epsilon}}. \end{aligned}
	\end{equation*}
	Note that 
	\begin{equation*}
		\begin{aligned}
			\frac{|I(\varepsilon_1,\cdots, \varepsilon_{n_k+N_k+(p-1)M})|}{|I(\varepsilon_1,\cdots, \varepsilon_n)|}&=\frac{Q_n}{Q_{n_k+N_k+(p-1)M}}\\
			&=\prod\limits_{m=n_k+N_k+(p-1)M+1}^{n}q_m\\
			&\le Q_{n}^{\eta M}= |I(\varepsilon_1,\cdots,\varepsilon_n)|^{-\eta M},
		\end{aligned}
	\end{equation*}
	then
	$$
	\mu\big(I(\varepsilon_1,\cdots, \varepsilon_n)\big)\le C\Big(|I(\varepsilon_1,\cdots,\varepsilon_n)|^{1-\eta M}\Big)^{\frac{
			1-3\epsilon}{1+\alpha+3\epsilon}}\le C|I(\varepsilon_1,\cdots,\varepsilon_n)|^{\frac{
			1-4\epsilon}{1+\alpha+3\epsilon}}.
	$$
	
	(VII) When $n_k+N_k+p_{k+1}M<n<n_{k+1},$ similar to (VI), we have 
	$$\mu\big(I(\varepsilon_1,\cdots, \varepsilon_n)\big)\le C|I(\varepsilon_1,\cdots,\varepsilon_n)|^{\frac{
			1-4\epsilon}{1+\alpha+3\epsilon}}.$$
	
	In summary, we have shown that for all cylinders $I(\varepsilon_1,\cdots,\varepsilon_n)$ with $n\ge n_{K}$,
	$$\mu\big(I(\varepsilon_1,\cdots, \varepsilon_n)\big)\le C|I(\varepsilon_1,\cdots,\varepsilon_n)|^{\frac{
			1-4\epsilon}{1+\alpha+3\epsilon}}.$$
	Therefore, by Lemma \ref{3.3}, we get
	$$\dim_H\mathcal{C}_{\infty}\ge \frac{
		1-4\epsilon}{1+\alpha+3\epsilon},$$
	by letting $\epsilon\rightarrow0$, we obtain that 
	$$\dim_H\mathcal{C}_{\infty}\ge \frac{
		1}{1+\alpha}.$$
		
\medskip
\noindent\textbf{Acknowledgment.}
The authors is supported by the National Natural Science Foundation of China (No. 12371086, 12271175) and the National Key R\&D Program of China (No. 2024YFA1013701).


{\small\begin{thebibliography}{99}
			\bibitem{Airey15}
			D. Airey and B. Mance, Unexpected distribution phenomenon resulting from Cantor series expansions, Adv. Math. {\bf 279} (2015), 372--404.
			
			\bibitem{Bandi23}
			P. Bandi, A. Ghosh and D. Nandi, Exact approximation order and well-distributed sets, Adv. Math. {\bf 414} (2023), Paper No. 108871, 19 pp.
			
			%			\bibitem{Beresnevich06}
			%			V.~V. Beresnevich, D. Dickinson and S.~L. Velani, Measure theoretic laws for lim sup sets, Mem. Amer. Math. Soc. {\bf 179} (2006), no.~846, x+91 pp.
			
			%			\bibitem{Beresnevich10}
			%			V.~V. Beresnevich and S.~L. Velani, Classical metric Diophantine approximation revisited: the Khintchine-Groshev theorem, Int. Math. Res. Not. (2010), no.~1, 69--86.
			
			%			\bibitem{Bovey86}
			%			J. Bovey and M.~M. Dodson, The Hausdorff dimension of systems of linear forms, Acta Arith. {\bf 45} (1986), no.~4, 337--358.
			\bibitem{Besicovitch34}
			A.~S. Besicovitch, Sets of Fractional Dimensions (IV): On Rational Approximation to Real Numbers, J. London Math. Soc. {\bf 9} (1934), no.~2, 126--131.
			
			\bibitem{Bugeaud03}
			Y. Bugeaud, Sets of exact approximation order by rational numbers, Math. Ann. {\bf 327} (2003), no.~1, 171--190.%; MR2006007
			
			\bibitem{Bugeaud08}
			Y. Bugeaud, Sets of exact approximation order by rational numbers. II, Unif. Distrib. Theory {\bf 3} (2008), no.~2, 9--20.
			
			\bibitem{Bugeaud12}
			Y. Bugeaud and C.~G.~T. de~Araujo~Moreira, Sets of exact approximation order by rational numbers III, Acta Arith. {\bf 146} (2011), no.~2, 177--193.
			
			\bibitem{Cantor69}
			G. Cantor, Ueber die einfachen Zahlensysteme. Z. Math. Phys. {\bf 14} (1869), 121--128.
			
			%			\bibitem{Dodson92}
			%			M.~M. Dodson, Hausdorff dimension, lower order and Khintchine's theorem in metric Diophantine approximation, J. Reine Angew. Math. {\bf 432} (1992), 69--76.
			
			\bibitem{Falconer14}
			K.~J. Falconer,  Fractal geometry, third edition, Wiley, Chichester, 2014.
			
			\bibitem{Fan09}
			A.~H. Fan et al., On Khintchine exponents and Lyapunov exponents of continued fractions, Ergodic Theory Dynam. Systems {\bf 29} (2009), no.~1, 73--109.
			
			\bibitem{Fang20}
			L. Fang, M. Wu and B. Li, Approximation orders of real numbers by $\beta$-expansions, Math. Z. {\bf 296} (2020), nos.~1-2, 13--40.%; MR4140729
			
			\bibitem{Fishman15}
			L. Fishman et al., Shrinking targets for nonautonomous dynamical systems corresponding to Cantor series expansions, Bull. Aust. Math. Soc. {\bf 92} (2015), no.~2, 205--213.
			
			\bibitem{Guting69}
			R. G\"uting, On Mahler's function $\theta \sb{1}$, Michigan Math. J. {\bf 10} (1963), 161--179.
			
			\bibitem{Jarnik29}
			V. Jarn\'ik, Diophantische Approximationen und Hausdorffsches Mass, Matem. Sb. {\bf 36} (1929), 371--382.
			
			
			\bibitem{Jarnik31}
			V. Jarn\'ik, \"Uber die simultanen diophantischen Approximationen, Math. Z. {\bf 33} (1931), no.~1, 505--543.%; MR1545226
			
			\bibitem{Khinchin24}
			A.~Y. Khinchine, Einige S\"atze \"uber Kettenbr\"uche, mit Anwendungen auf die Theorie der Diophantischen Approximationen, Math. Ann. {\bf 92} (1924), nos.~1-2, 115--125.%; MR1512207
			
			\bibitem{Ma26}
			C. Ma et al., Approximation orders of real numbers in Cantor series expansions, J. Math. Anal. Appl. {\bf 553} (2026), no.~2, Paper No. 129924, 15 pp.
			
			\bibitem{Mattila95}
			P. Mattila, Geometry of sets and measures in Euclidean spaces, Cambridge Studies in Advanced Mathematics, 44, Cambridge Univ. Press, Cambridge, 1995.
			
			
			\bibitem{Sun17}
			Y. Sun and C. Cao, Dichotomy law for shrinking target problems in a nonautonomous dynamical system: Cantor series expansion, Proc. Amer. Math. Soc. {\bf 145} (2017), no.~6, 2349--2359.
			
			
			\bibitem{Zhang24}
			X. Zhang and W. Zhong, Exact Diophantine approximation of real numbers by $\beta$-expansions, Discrete Contin. Dyn. Syst. {\bf 44} (2024), no.~9, 2684--2696.%; MR4762614
			
			\vskip 4pt
			\vskip 15pt
	\end{thebibliography}}
\end{document}